\documentclass[12pt,reqno]{amsart}

\usepackage{amsthm,amssymb}
\usepackage{hyperref}
\usepackage{todonotes}

\numberwithin{equation}{section}
\newtheorem{theorem}{Theorem}[section]
\newtheorem{lemma}[theorem]{Lemma}
\newtheorem{corollary}[theorem]{Corollary}
\newtheorem{proposition}[theorem]{Proposition}

\theoremstyle{definition}
\newtheorem{definition}[theorem]{Definition}
\newtheorem{construction}[theorem]{Construction}
\newtheorem{notation}[theorem]{Notation}
\newtheorem{example}[theorem]{Example}
\theoremstyle{remark}
\newtheorem{remark}[theorem]{Remark}
\newtheorem*{remark*}{Remark}

\newcommand{\field}[1]{\ensuremath{\mathbb{#1}}}

\newcommand{\rational}{\field{Q}}
\newcommand{\integer}{\field{Z}}

\DeclareMathOperator{\lcm}{lcm}

\newcommand{\invop}[1]{\mathbin{\overline{#1}}}
\newcommand{\bstar}{\invop{*}}
\newcommand{\rt}{\vartriangleright}
\newcommand{\lt}{\vartriangleleft}
\newcommand{\rtb}{\blacktriangleright}
\newcommand{\ltb}{\blacktriangleleft}
\newcommand{\zttinv}{\ensuremath{\integer[t,t^{-1}]}}
\newcommand{\zt}{\ensuremath{\integer[t]}}
\newcommand{\binseq}{\ensuremath{\mathcal{B}}}
\newcommand{\avgquandle}[2][\rational]{\ensuremath{#1_{#2}}}
\newcommand{\rght}[1]{\rho_{#1}}

\begin{document}
\title[The rack congruence condition]{The rack congruence condition \\ and half congruences in racks}

\author{Wayne Burrows}
\address{School of Mathematical and Computational Sciences, Massey University,
         Private Bag 11 222, Palmerston North 4442, New Zealand}
\email{wjburrows@gmail.com}

\author{Christopher Tuffley}
\address{School of Mathematical and Computational Sciences, Massey University,
         Private Bag 11 222, Palmerston North 4442, New Zealand}
\email{c.tuffley@massey.ac.nz}


\begin{abstract}
Racks and quandles are algebraic structures with a single binary operation that is right self-distributive and right invertible, and additionally idempotent in the case of quandles. The invertibility condition is equivalent to the existence of a second binary operation that acts as a right inverse to the first, so that racks and quandles may also be viewed as algebraic structures with a pair of (dependent) binary operations.

When forming a quotient rack or quandle it is necessary to take this two-operation view, and define a \emph{congruence} as an equivalence relation on the rack or quandle that respects both operations. However, in defining a congruence some authors have omitted the condition on the inverse operation, and defined a congruence as an equivalence relation respecting the primary operation only. We show that while respecting the primary operation is sufficient in the case of finite racks and quandles, it is not in general sufficient in the infinite case. We do this by constructing explicit examples of ``half congruences'': equivalence relations that respect exactly one of the two operations. 

Our work also allows us to completely characterise congruences in connected Alexander quandles.
\end{abstract}

\maketitle

\section{Introduction}

A rack $(R,\rt)$ is an algebraic structure consisting of a set $R$ together with a single binary operation $\rt$ on $R$ that is right self-distributive and right invertible: for all $x,y,z\in R$ we have
\begin{equation}
\label{eq:rightselfdistributivity}
(x\rt y)\rt z = (x\rt z)\rt (y\rt z),
\end{equation}
and for all $y,z\in R$ there exists a unique $x\in R$ such that $x\rt y=z$. A quandle is a rack in which the binary operation is also idempotent: for all $x\in R$ we have $x\rt x=x$. The quandle axioms can be viewed as axiomatising the properties of conjugation in groups~\cite{joyce-1982}, and as algebraically encoding the Reidemeister moves in knot theory.
See \cite{carter-2012,elhamdadi-2020,fenn-rourke-1992}
for a history and survey of racks and quandles.

Right invertibility is equivalent to the existence of a second binary operation $\lt$ that acts as a right inverse to $\rt$, in the sense that for all $x,y\in R$ we have
\begin{align}
\label{eq:inverseop}
(x\rt y)\lt y &= x, &
(x\lt y)\rt y &= x.
\end{align}
Indeed, some authors such as Joyce~\cite{joyce-1982} define a quandle as a set with two binary operations $\rt$ and $\lt$ satisfying~\eqref{eq:rightselfdistributivity},~\eqref{eq:inverseop} and the idempotence axiom $x\rt x=x$ for all $x$. The inverse operation turns out to be right self-distributive as well, and so $(R,\lt)$ is also a rack, and a quandle when $(R,\rt)$ is.

The inverse operation plays an important role in the theory: for example, when defining the orbit of an element it is necessary to consider the effect of both rack operations. However, the view of racks and quandles as algebraic structures with a single binary operation appears on occasion to have led to the inverse operation being omitted from statements where it needs to appear. One place where this has sometimes occurred is in the definition of subracks and subquandles; a second is in the definition of congruences, which are equivalence relations used to construct quotient racks and quandles.
It is this second setting that is our focus in this paper.

In the familiar settings of groups and rings we form quotients by starting with a suitable substructure: a normal subgroup in the case of groups, and an ideal in the case of rings. The substructure is used to define an equivalence relation on the group or ring, and the elements of the quotient structure are the equivalence classes of this relation --- the cosets of the chosen substructure. The conditions imposed on the substructure ensure that the induced operations on equivalence classes are well defined. 
In contrast, in the setting of racks and quandles quotients do not involve a choice of substructure, and instead we must work directly with suitably chosen equivalence relations. A necessary first step then is to determine what conditions must be imposed on the equivalence relation in order for the quotient to again be a rack or quandle. 

The required condition is that the equivalence relation must respect both rack operations, in the sense that
\[
\text{$a\sim c$ and $b\sim d$} \qquad\Rightarrow\qquad\text{$a\rt b\sim c\rt d$ and $a\lt b\sim c\lt d$}.
\]
An equivalence relation satisfying this condition is called a \emph{rack congruence}. While some authors such as Bonatto and Stanovsk\'y~\cite{bonatto-stanovsky-2021} have required a rack congruence
to respect both operations, others 
such as Fenn and Rourke~\cite{fenn-rourke-1992} and Ryder~\cite{ryder-1993,ryder-1995} have omitted the condition on the inverse rack operation, and defined a rack congruence as an equivalence relation respecting the primary rack operation $\rt$ only. Our goal in this paper is to clarify this situation. We show that, while respecting the primary rack operation only is sufficient for finite racks, it is not in general sufficient for infinite racks. We do this by constructing explicit examples of ``half congruences'': equivalence relations that respect exactly one of the two rack operations.

\subsection{Discussion}

A similar situation to the one seen here arises in the context of quasigroups. These are commonly defined as a set $Q$ with a single binary operation $*$ (multiplication) such that for all $a,b\in Q$, the equations $a*x=b$ and $y*a=b$ have unique solutions $x$ and $y$ respectively. 
Equivalently, a quasigroup is a set $Q$ with a binary operation $*$ such that all left translation maps $x\mapsto a*x$ and right translation maps $x\mapsto x*a$ are bijections.
With this definition, Bates and Kiokemeister~\cite{bates-kiokemeister-1948} constructed a homomorphism from a quasigroup $(Q,*)$ such that the image is not a quasigroup. This is equivalent to the existence of a congruence on $Q$ (that is, an equivalence relation on $Q$ that respects $*$) such that the quotient is not a quasigroup. 

The underlying issue is that the definition above does not define quasigroups equationally, that is, through a set of identities. Evans~\cite{evans-1949} showed how to define quasigroups equationally, as a set with three binary operations, multiplication, left division, and right division, that satisfy a series of identities. Evans showed moreover that the image of a homomorphism respecting all three operations is always a quasigroup, but the image of a homomorphism respecting only two of them is not necessarily a quasigroup.

In universal algebra a \emph{variety} is a class of algebras that is closed under homomorphic images, subalgebras, and direct products, and an \emph{equational class} is a class of algebras that is defined through a set of identities. Every equational class is a variety, and by a theorem of Birkhoff~\cite{birkhoff-1935} the converse is also true. For an algebra $A$ belonging to an equational class $\mathcal{K}$, any quotient $A/{\sim}$ formed using a equivalence relation $\sim$ respecting all of the operations in $A$ will be a homomorphic image of $A$, and hence also a member of the class $\mathcal{K}$.
See for example Bergman~\cite{bergman-universal-algebra-2012} for an introduction to the fundamentals of universal algebra.

In the case of racks and quandles, the two-operation approach to defining them is an equational definition, but the one-operation approach is not, because the right invertibility axiom does not take the form of an identity. Thus we know \emph{a priori} that the quotient of a rack or quandle by an equivalence relation respecting both operations will again be a rack or quandle. However, we have no prior reason to expect the same to be true of a quotient by an equivalence relation that is required to respect the primary operation only. Thus, it may not be sufficient to define rack and quandle congruences as equivalence relations that respect the primary operation, without requiring that they respect the inverse operation also. The primary goal of this paper is to show by example that in general it is necessary to require a rack or quandle congruence to respect both operations.

The quasigroup counterexamples presented by Bates and Kiokemeister~\cite{bates-kiokemeister-1948} and Evans~\cite{evans-1949} are nonconstructive. They define a partial multiplication structure on a finite set, and obtain the desired algebraic structure through a chain of extensions. The resulting examples are necessarily infinite, for the same underlying reason that applies in our work here: no finite structure can be a counterexample, because the left and right translation maps in the image of the quasigroup are always onto, and a map from a finite set to itself that is onto is necessarily one-to-one. In contrast to their work, the rack and quandle examples we present are entirely constructive: we will specify the underlying sets and binary operations completely explicitly.

\subsection{Organisation}

The paper is organised as follows. In Section~\ref{sec:definitions} we set out the basic definitions of racks and quandles; in Section~\ref{sec:congruences} we develop the theory of rack congruences and quotient racks and quandles; and in Section~\ref{sec:homomorphisms} we examine rack homomorphisms and establish the connection with congruences and quotients, proving a First Isomorphism Theorem for racks and quandles, analogous to the corresponding theorems in group and ring theory. The definitions and results in these sections are for the most part not new, and many of the results follow readily from standard results in universal algebra. Nevertheless, we feel there is a benefit to including them. Firstly, doing so makes the paper more self-contained. But secondly and more importantly, racks and quandles are a relatively young area of study, and at present the literature on them is largely confined to papers rather than textbooks\footnote{At present we are only aware of one textbook on racks and quandles, Elhamdadi and Nelson~\cite{elhamdadi-nelson-2015}, and it does not address rack and quandle congruences.}. Moreover it appears that simple results have often been stated without proof --- perhaps because we ``know'' what should be true from our experience with groups and rings --- and this appears to have led at times to incorrect statements being made and then repeated. That has occurred with the definition of congruences, and so we think there is a benefit to the basic theory of rack congruences, quotients and homomorphisms being set down all in one place, with proof, in an elementary and self-contained form.

We then turn to proving our original results in Sections~\ref{sec:existence} and~\ref{sec:weightedaverage}. In Section~\ref{sec:existence} we give our first examples of half congruences in racks and quandles, including an example in a finitely presented quandle. In Section~\ref{sec:weightedaverage} we then characterise congruences in ``weighted average quandles on $\rational$'': quandles with underlying set $\rational$, where the quandle operation is a weighted average with some fixed weight $\tau$. This turns out to be an interesting family of examples, exhibiting all four possible values of the truth table for the two statements ``$\rt$-congruence implies $\lt$-congruence'' and ``$\lt$-congruence implies $\rt$-congruence''.

The weighted average quandles in Section~\ref{sec:weightedaverage} are special cases of a more general construction known as \emph{Alexander quandles}. We conclude the paper in Section~\ref{sec:generalAlexanderquandles} with a brief discussion of the extent to which our results on weighted average quandles on $\rational$ do and do not extend to Alexander quandles in general. The extension yields a complete characterisation of quandle congruences in connected Alexander quandles.

\section{Magmas, racks and quandles}
\label{sec:definitions}

In this section we define racks and quandles and give some examples that we will use later in the paper.  In both this section and the next we will work at the level of magmas, so that we can consider each of the rack and quandle properties separately. A magma is simply a set with a binary operation, that need not satisfy any further condition:

\begin{definition}
A \textbf{magma} is a pair $(M,*)$ such that $M$ is a set and $*$ is a binary operation on $M$, that is, a function $*:M\times M\to M$.
\end{definition}

\begin{definition}
Let $(M,*)$ be a magma. Then $*$ is 
\begin{description}
\item[idempotent] if for all $x\in M$,
\[
x*x = x.
\]

\item[right invertible] if for all $y,z\in M$, there exists a unique $x\in M$ such that
\[
x*y = z.
\]

\item[right self-distributive] if for all $x,y,z\in M$,
\[
(x*y)*z = (x*z)*(y*z).
\]
\end{description}
\end{definition}

\begin{definition}
A \textbf{rack} is a magma $(R,\rt)$ such that $\rt$ is right invertible and right self-distributive.

A \textbf{quandle} is a magma $(Q,\rt)$ such that $\rt$ is idempotent, right invertible and right self-distributive.
Note that every quandle is a rack.
\end{definition}
In a rack we regard the product $x\rt y$ as the result of $y$ acting on $x$.

\subsection{Inverse operation formulation of right invertibility}

Right invertibility is equivalent to the existence of a second binary operation that acts as a right inverse to the first:
\begin{lemma}
\label{lem:invop}
Let $(M,*)$ be a magma. Then
$*$ is right invertible if and only if there exists a binary operation $\bstar$ on $M$ such that
\begin{equation}
\label{eq:invop}
(x\bstar y)* y = x \qquad \text{and} \qquad (x*y)\bstar y = x,
\end{equation}
for all $x,y\in M$. 
\end{lemma}
Note that the binary operation $\bstar$ is also right invertible, with $\invop{\bstar}=*$. Moreover, if $*$ is idempotent then so is $\bstar$.

\begin{proof}
Suppose first that $*$ is right invertible. 
For $y\in M$ define the right translation $\rght{y}:M\to M$ by
\[
x\rght{y} = x*y,
\]
where we have chosen to write $\rght{y}$ on the right of its argument. In the context of quandles this map is sometimes known as the \emph{symmetry at $y$}, and denoted $S_y$. Then since $*$ is right invertible it follows that $\rght{y}$ is a bijection for each $y\in M$, so we may define a second binary operation $\bstar$ on $M$ by
\[
x\bstar y = x\rght{y}^{-1}.
\]
Here as usual $\rght{y}^{-1}:M\to M$ is the inverse function to $\rght{y}$, such that both $\rght{y}\circ \rght{y}^{-1}$ and $\rght{y}^{-1}\circ \rght{y}$ are the identity function on $M$. 
Then
\begin{align*}
(x\bstar y)* y &= (x\rght{y}^{-1})\rght{y} = x, \\
(x* y)\bstar y &= (x\rght{y})\rght{y}^{-1} = x, 
\end{align*}
as required. 

Now suppose that there exists a binary operation $\bstar$ satisfying~\eqref{eq:invop}. Given $y,z\in M$, let $x=z\bstar y$. Then
\[
x*y = (z\bstar y)*y = z,
\]
so the equation $x*y=z$ has a solution $x\in M$. Moreover, if $w\in M$ also satisfies $w*y=z$ then
\[
w = (w*y)\bstar y = z\bstar y = (x*y)\bstar y = x,
\]
so the solution to this equation is unique. It follows that $*$ is right invertible, as required.
\end{proof}

We show by example that both conditions of~\eqref{eq:invop} must be satisfied in order for $*$ to be right invertible. 

\begin{example}
Take $M=\integer$, and define binary operations $*_1$, $*_2$ on \integer\ by 
\begin{align*}
x *_1 y &= 2x,  & x *_2 y &= \lfloor x/2\rfloor.  
\end{align*}
Then neither operation is right invertible: $*_1$ is not right invertible, because if $z$ is odd then the equation $x*_1y=z$ has no solution $x$; and $*_2$ is not right invertible, because for all $y,z$ the equation $x*_2y=z$ has two solutions $x=2z$ and $x=2z+1$. However 
\[
(x *_1 y) *_2 y = x,
\]
while 
\[
(x *_2 y) *_1 y = 2\lfloor x/2\rfloor 
   = \begin{cases} 
   x & \text{$x$ even}, \\
   x-1 & \text{$x$ odd}.
  \end{cases}
\]
Taking $*=*_1$, $\bstar=*_2$ shows that $(x* y)\bstar y = x$ is insufficient for $*$ to be right invertible; and taking $*=*_2$, $\bstar=*_1$ shows that $(x\bstar y)* y = x$ is insufficient for $*$ to be right invertible.
\end{example}

\begin{definition}
Let $(R,\rt)$ be a rack, with inverse operation $\invop\rt$. Then as we show below $\invop\rt$ is also right self-distributive, and consequently  $(R,\invop{\rt})$ is also a rack, the \textbf{dual rack} of $(R,\rt)$. When $(R,\rt)$ is a quandle so is $(R,\invop{\rt})$, the \textbf{dual quandle} of $(R,\rt)$.
\end{definition}

\begin{notation}
Let $(R,\rt)$ be a rack. 
There are various notations and conventions in use for the inverse operation $\invop\rt$. Other notations include $\rt^{-1}$ and $\lt$; and $\lt$ is sometimes defined (see for example Crans~\cite{crans-2004}) so that it is left invertible and left self-distributive, by defining $\lt$ so that
\[
y\lt x = x \invop\rt y.
\]
To place the two operations on an equal footing, in this paper we will use the notation $\lt$, and we define $\lt$ so that it is equal to ${\invop\rt}$, making it right invertible and right self-distributive.
\end{notation}

\begin{proof}[Proof that the inverse rack operation is right self-distributive]
To show that the inverse rack operation $\lt$ is right self-distributive we may proceed as follows. Given $x,y,z\in R$ let
\begin{align*}
v &= y\lt z, \\
u &= (x\lt z)\lt v = (x\lt z)\lt (y\lt z).
\end{align*}
Rearranging these equations we obtain $y=v\rt z$ and $x\lt z=u\rt v$, and so $x= (u\rt v)\rt z$. Then
\begin{align*}
(x\lt y)\lt z &= \Bigl(\bigl((u\rt v)\rt z\bigr)\lt(v\rt z)\Bigr)\lt z \\
              &= \Bigl(\bigl((u\rt z)\rt (v\rt z)\bigr)\lt(v\rt z)\Bigr)\lt z \\
              &= (u\rt z)\lt z \\
              &= u.
\end{align*}
Hence $(x\lt y)\lt z = (x\lt z)\lt (y\lt z)$, 
as required.
\end{proof}

It can further be shown that the two rack operations are right distributive over each other:
\begin{lemma}
\label{lem:distributivity}
Let $(R,\rt)$ be a rack, with inverse rack operation $\lt$. Then $\rt$ and $\lt$ are right distributive over each other: for all $x,y,z\in R$ we have
\begin{align}
(x\rt y)\lt z &= (x\lt z)\rt(y\lt z), \label{eq:lt-over-rt} \\
(x\lt y)\rt z &= (x\rt z)\lt(y\rt z). \label{eq:rt-over-lt}
\end{align}
\end{lemma}

\begin{proof}
To prove~\eqref{eq:lt-over-rt} we act on each side of the equation from the right by $z$ using $\rt$. For the left hand side we have
\[
\bigl((x\rt y)\lt z\bigr)\rt z = x\rt y,
\]
while for the right hand side we have
\begin{align*}
\bigl((x\lt z)\rt(y\lt z)\bigr) \rt z
    &= \bigl((x\lt z)\rt z\bigr)\rt\bigl((y\lt z)\rt z\bigr) \\
    &= x\rt y.        
\end{align*}
Since the results are equal, equation~\eqref{eq:lt-over-rt} holds by right invertibility of $\rt$. Equation~\eqref{eq:rt-over-lt} follows by a similar argument, or by applying the result just proved in the dual rack $(R,\lt)$.
\end{proof}

\subsection{Examples}

We list some racks and quandle constructions that will arise in the remainder of the paper.

\begin{definition}[Constant action racks]
Let $R$ be a set, let $f:R\to R$ be a bijection, and define the binary operation $\rt$ on $R$ by
\[
x\rt y = f(x)
\]
for all $x,y\in R$. Then it is easily checked that $(R,\rt)$ is a rack, 
with inverse rack operation $\lt$ given by
\[
x\lt y = f^{-1}(x)
\]
for all $x,y\in R$.
Such a rack is known as a \textbf{constant action rack} or \textbf{permutational rack}. 

A constant action rack is a quandle if and only if the underlying bijection $f$ is the identity. In that case the quandle operation is given by $x\rt y=x$ for all $x,y\in R$; such a quandle is called a \textbf{trivial quandle} or \textbf{projection quandle}.
\end{definition}

\begin{example}
\label{ex:leftshift}
Let $\binseq=\{0,1\}^\integer = \bigl\{a=(a_i)_{i\in\integer}:a_i\in\{0,1\}\bigr\}$ be the set of bi-infinite binary sequences, and let $\ell,r:\binseq\to\binseq$ be the left- and right-shift operators defined by
\begin{align*}
(\ell(a))_i &= a_{i+1} && \text{(left-shift)} \\
(r(a))_i &= a_{i-1} && \text{(right-shift)}.
\end{align*}
Then $\ell$ is a bijection with inverse $r$, and so defines a constant action rack structure on $\binseq$. We will write $\binseq_\ell$ for this rack.
\end{example}

\begin{definition}[Alexander quandles]
\label{defn:alexander}
Let $A$ be a module over the Laurent ring \zttinv, and define the binary operation $\rt$ on $A$ by
\[
x\rt y = tx +(1-t)y
\]
for all $x,y\in A$. Then $(A,\rt)$ is a quandle, with inverse quandle operation given by
\[
x\lt y = t^{-1}x +(1-t^{-1})y
\]
for all $x,y\in A$. Such a quandle is known as an \textbf{Alexander quandle} or \textbf{affine quandle}.

The dual quandle of an Alexander quandle is also an Alexander quandle, because the map $\zttinv\to\zttinv$ induced by $t\mapsto t^{-1}$ is a ring isomorphism.
\end{definition}

We note that a \zttinv-module structure on an abelian group $A$ corresponds to a choice of group isomorphism $\phi:A\to A$, with the action of $t$ on $x\in A$ then given by $tx=\phi(x)$.

\begin{construction}[Weighted average quandles]
\label{cons:weightedaverage} 
As a special case of Definition~\ref{defn:alexander}, let $R$ be a ring with unity and let $\tau\in R$ be a unit. Then $\phi: R\to R$ defined by 
\[
\phi(x) = \tau x
\]
is a group isomorphism of the additive group $(R,+)$, and so defines a \zttinv-module structure on $R$ and hence an Alexander quandle structure on $R$ with quandle operation
\[
x\rt y = \tau x + (1-\tau)y.
\]
The quandle operation may be viewed as a weighted average, and we will refer to $(R,\rt)$ as a \textbf{weighted average quandle on $R$}, with weight $\tau$. We will denote this quandle by \avgquandle[R]{\tau}.

Note that if $\tau=1$ (the unity of $R$) then the resulting weighted average quandle is the trivial quandle on $R$. We will say that \avgquandle[R]{\tau}\ is a \textbf{nontrivial weighted average quandle on $R$} if $\tau\neq 1$.
\end{construction}

We will make use of weighted average quandles on \rational\ in Section~\ref{sec:weightedaverage}.

One other case of Construction~\ref{cons:weightedaverage} deserves special mention:
\begin{example}[Dihedral or cyclic quandles]
Let $R$ be $\integer$ or $\integer_n$, and take $\tau=-1$. Then
\[
x\rt y = \tau x + (1-\tau)y = 2y-x
\]
defines a quandle structure on $R$. These quandles are known as \textbf{cyclic} or \textbf{dihedral} quandles.
\end{example}

Dihedral quandles are examples of \textbf{involutory quandles}: quandles in which $x\rght{y}^2=(x\rt y)\rt y=x$ for all $x$ and $y$. Involutory quandles are also known as \textbf{$2$-quandles}; and more generally, we say that a rack $R$ is an \textbf{$n$-rack} if $x\rght{y}^n=x$ for all $x,y\in R$. Observe that if $R$ is finite with $|R|=r$, then (since each symmetry $\rght{y}$ can be regarded as an element of the symmetric group $S_R$, and so has order dividing $|S_r|=r!$ by Lagrange's Theorem) $R$ is an $n$-rack for some $n\leq r!$.

\subsection{Subracks and subquandles}

As usual, a subrack of a rack $R$ is a subset $S$ of $R$ that is itself a rack, using the rack operation in $R$; and similarly for subquandles. Subracks and subquandles may be characterised as follows:

\begin{proposition}
\label{prop:subracktest}
Let $(R,\rt)$ be a rack (resp.\ quandle), and let $S$ be a nonempty subset of $R$. Then $(S,\rt)$ is a subrack (resp.\ subquandle) of $R$ if and only if it is closed with respect to both rack operations $\rt$ and $\lt$.
\end{proposition}

\begin{proof}
Suppose that $S$ is closed with respect to both rack operations $\rt$ and $\lt$. Then $\rt$ is a binary operation on $S$, and it is right self-distributive on $S$ because it is right self-distributive on $R$. Moreover, if $\rt$ is idempotent on $R$ then it is idempotent on $S$ also.

To complete the proof that $S$ is a subrack/subquandle of $R$ it thus remains to show that $\rt$ is right invertible on $S$. Given $y,z\in S$ observe that the equation $x\rt y=z$ has the unique solution $x=z\lt y$ in $R$, and $x\in S$, because $S$ is closed with respect to $\lt$. 

Conversely, suppose that $S$ is a subrack/subquandle of $R$. Then the restrictions of $\rt$ and $\lt$ to $S$ are binary operations on $S$, and so $S$ is closed with respect to both $\rt$ and $\lt$.
\end{proof}

In stating the condition for $S$ to be a subrack of $R$ some authors have omitted the requirement that $S$ be closed with respect to the inverse rack operation, and required closure with respect to the primary rack operation $\rt$ only. This is sufficient if $R$ is finite (or more generally, if $R$ is an $n$-rack), because in that setting closure with respect to $\lt$ can be deduced for example from repeated application of $\rt$. However, in general it is necessary to check closure with respect to both rack operations. 
Kamada~\cite{kamada-2010} gives several quandle examples demonstrating this; and a simple example is given by the subset $\rational^+=\{x\in\rational:x>0\}$ in the weighted average quandle $\avgquandle{1/2}$.
This is not a subquandle, because for example the equation $x\rt 2=\frac{x+2}2=1$ has no solution $x\in\rational^+$.

\section{Congruences and quotients}
\label{sec:congruences}

In this section we characterise when an equivalence relation on a rack or quandle 
induces a rack or quandle structure on the quotient, that is, on the set of equivalence classes. We will continue to work at the level of magmas, so that we can consider each of the rack and quandle properties separately. The key to the characterisation is the notion of a \emph{congruence}:

\begin{definition}
Let $(M,*)$ be a magma. An equivalence relation $\sim$ on $M$ is a $*$-\textbf{congruence} on $M$ if it satisfies the following \textbf{congruence condition}: for all $a,b,c,d\in M$,
\[
a\sim c,\; b\sim d \qquad \Rightarrow\qquad a*b\sim c*d.
\]
\end{definition}

\begin{notation}
Let $X$ be a set and let $\sim$ be an equivalence relation on $X$. Given $x\in X$ we write $[x]$ for the equivalence class of $x$,
\[
[x] = \{y\in X: x\sim y\},
\]
and $X/{\sim}$ for the set of equivalence classes of $\sim$,
\[
X/{\sim} = \{[x]:x\in X\}.
\]
\end{notation}

\begin{lemma}
\label{lem:welldefinedinducedop}
Let $(M,*)$ be a magma, and let $\sim$ be an equivalence relation on $M$. Then the induced operation $*$ on $M/{\sim}$ defined by 
\[
[x]*[y] = [x*y]
\]
is well defined if and only if $\sim$ is a $*$-congruence.
\end{lemma}

In view of Lemma~\ref{lem:welldefinedinducedop}, if $\sim$ is a $*$-congruence on $M$ then $(M/{\sim},*)$ is also a magma, the \textbf{quotient magma of $M$ by $\sim$}. 

\begin{proof}
Suppose that $\sim$ is a $*$-congruence. Let $a,b,c,d\in M$ be such that $[a]=[c]$ and $[b]=[d]$. Then $a\sim c$ and $b\sim d$ so we have also $a*b\sim c*d$, and therefore $[a*b]=[c*d]$. This shows that the induced operation on equivalence classes is well defined.

Suppose that $\sim$ is not a $*$-congruence. Then there exist elements $a,b,c,d\in M$ such that $a\sim c$, $b\sim d$, but $a*b\not\sim c*d$. Expressing these statements in terms of congruence classes we have  $[a]=[c]$, $[b]=[d]$ but $[a*c]\neq[c*d]$, so the induced operation $[x]*[y]=[x*y]$ is not well defined.
\end{proof}

We show that if $*$ is idempotent or right self-distributive on $M$ and $\sim$ is a $*$-congruence, then the induced operation in the quotient $M/{\sim}$ inherits this property too.

\begin{lemma}
\label{lem:inducedopprops}
Let $(M,*)$ be a magma, and let $\sim$ be a $*$-congruence on $M$.
\begin{enumerate}
\item
If $*$ is idempotent on $M$ then the induced operation $*$ is idempotent on $M/{\sim}$. 
\item
If $*$ is right self-distributive on $M$ then the induced operation $*$ is right self-distributive on $M/{\sim}$. 
\end{enumerate}
\end{lemma}

\begin{proof} ~
\begin{enumerate}
\item
Suppose that $*$ is idempotent on $M$. Then for all $x\in M$ we have
\[
[x]*[x] = [x*x] = [x],
\]
so $*$ is idempotent on $M/{\sim}$.

\item
Suppose that $*$ is right self-distributive on $M$. Then for all $x,y,z\in M$ we have
\begin{align*}
([x]*[y])*[z] &= [x*y]*[z] \\
              &= [(x*y)*z] \\
              &= [(x*z)*(y*z)] \\
              &= [x*z]*[y*z] \\
              &= ([x]*[z])*([y]*[z]), 
\end{align*}
so $*$ is right self-distributive on $M/{\sim}$.
\end{enumerate}
\end{proof}

We now consider the question of when the induced operation $*$ on $M/{\sim}$ is right invertible, given that $*$ is right invertible on $M$. We first note that right invertibility on $M$ guarantees that for all $y,z\in M$, the equation 
$[x]*[y]=[z]$ has a solution $[x]\in M/{\sim}$: letting $x=z\bstar y$ we
have
\[
[x]*[y] = [z\bstar y]*[y] = [(z\bstar y)*y] = [z],
\]
as required. We show however that the solution is unique if and only if $\sim$ is a $\bstar$-congruence.

\begin{lemma}
\label{lem:inducedinvertibility}
Let $(M,*)$ be a magma such that $*$ is right invertible, and let $\sim$ be a $*$-congruence on $M$. 
Then the induced operation $*$ on $M/{\sim}$ is right invertible if and only if $\sim$ is a $\bstar$-congruence. Moreover, when $*$ is right invertible  on $M/{\sim}$, the inverse operation is the induced operation
\[
[x]\bstar[y] = [x\bstar y].
\]
\end{lemma}

\begin{proof}
Suppose first that $\sim$ is a $\bstar$-congruence. Then by Lemma~\ref{lem:welldefinedinducedop} the induced operation
\[
[x]\bstar[y] = [x\bstar y]
\]
is well defined on $M/{\sim}$; and we have
\begin{align*}
([x]*[y])\bstar [y] &= [x*y]\bstar[y] = [(x*y)\bstar y] = [x], \\
([x]\bstar[y])* [y] &= [x\bstar y]*[y] = [(x\bstar y)* y] = [x].
\end{align*}
By Lemma~\ref{lem:invop} $*$ is right invertible on $M/{\sim}$.

Now suppose that $\sim$ is not a $\bstar$-congruence. Then there exist elements $a,b,c,d\in M$ such that $a\sim c$, $b\sim d$ but $a\bstar b\not\sim c\bstar d$. Let $[y]=[b]=[d]$, $[z]=[a]=[c]$, and consider the equation $X*[y]=[z]$ in $M/{\sim}$. On the one hand, we may regard this as $X*[b]=[a]$, so $X=[a\bstar b]$  is a solution; and on the other, we may regard it as $X*[d]=[c]$, so $X=[c\bstar d]$ is a solution. However $[a\bstar b]\neq [c\bstar d]$ by choice of $a,b,c,d$, so the solution to this equation is not unique. It follows that $*$ is not right invertible on $M/{\sim}$.
\end{proof}

The results above give us the following characterisation of equivalence relations on a rack for which the quotient is again a rack:

\begin{proposition}
\label{prop:quotientrack}
Let $(R,\rt)$ be a rack (resp.\ quandle), and let $\sim$ be an equivalence relation on $R$. Then $(R/{\sim},\rt)$ is a rack (resp.\ quandle) if and only if $\sim$ is both a $\rt$-congruence and a $\lt$-congruence.
\end{proposition}

\begin{proof}
The proposition is a corollary of the lemmas proved above. By Lemma~\ref{lem:welldefinedinducedop} $\rt$ is well defined on $R/{\sim}$ if and only if $\sim$ is a $\rt$-congruence. By Lemma~\ref{lem:inducedopprops} when the induced operation is well defined it is right self-distributive on $R/{\sim}$, and idempotent on $R/{\sim}$ when it is idempotent on $R$. By Lemma~\ref{lem:inducedinvertibility} $\rt$ is right invertible on $R/{\sim}$ if and only if $\sim$ is a $\lt$-congruence, completing the proof.
\end{proof}

In view of Proposition~\ref{prop:quotientrack} we make the following definition. 

\begin{definition}
Let $(R,\rt)$ be a rack, and let $\sim$ be an equivalence relation on $R$. Then $\sim$ is a \textbf{rack congruence} if and only if it is both a $\rt$-congruence and a $\lt$-congruence.

When $(R,\rt)$ is a quandle, we may also refer to a rack congruence on $R$ as a quandle congruence.
\end{definition}

We further make the following definition:

\begin{definition}
Let $(R,\rt)$ be a rack, and let $\sim$ be an equivalence relation on $R$. Then $\sim$ is a \textbf{half congruence} if it respects exactly one of the two rack operations: that is, if it is a $\rt$-congruence but not a $\lt$-congruence, or a $\lt$-congruence but not a $\rt$-congruence.
\end{definition}

A \textbf{shelf} is a magma $(S,\rt)$ such that $\rt$ is right self-distributive, and a \textbf{spindle} is a shelf $(S,\rt)$ such that $\rt$ is also idempotent~\cite[Sec.~3.1]{crans-2004}. By Lemma~\ref{lem:inducedopprops} the quotient of a rack by a half congruence has the structure of a shelf, and the quotient of a quandle by a half congruence has the structure of a spindle. The binary operation in the quotient depends on whether the half congruence respects the primary rack operation or the inverse rack operation.

\begin{remark}
Let $(G,*)$ be a group with identity element $1$, and let $\sim$ be a $*$-congruence on $G$. Then it is a simple exercise to show that $N=[1]$ is a normal subgroup of $G$, and $a\sim b$ if and only if $a^{-1}b\in N$, which of course holds if and only if $aN=bN$. It follows too that $a\sim b$ implies $a^{-1}\sim b^{-1}$. 

Similarly, if $(R,+,\times)$ is a ring and $\sim$ is both a $+$-congruence and a $\times$-congruence on $R$, then it is an easy exercise to show that $I=[0]$ is an ideal of $R$, and $x\sim y$ if and only if $x-y\in I$, which of course holds if and only if $x+I=y+I$.
\end{remark}

\subsection{A sufficient condition on a rack for every $\rt$-congruence to be a $\lt$-congruence}

If $R$ is a rack and $\sim$ is a $\rt$-congruence on $R$ then for all $y,z\in R$ the equation $[x]\rt[y]=[z]$ has at least one solution $[x]\in R/{\sim}$, namely $[x]=[z\lt y]$. Equivalently, for all $[y]\in R/{\sim}$ the map from $R/{\sim}$ to itself given by $[x]\mapsto [x]*[y]$ is onto. 
If $R/{\sim}$ is finite then the solution to this equation is necessarily unique, because a map from a finite set to itself is one-to-one if and only if it is onto. It follows that if $R$ is a finite rack then every $\rt$-congruence on $R$ is a $\lt$-congruence. In fact a much weaker finiteness condition is sufficient for this conclusion to hold, as we show below.

Let $(R,\rt)$ be a rack. Then for each $y\in R$ the right translation $\rght{y}:R\to R$ defined by
\[
x\rght{y} = x\rt y
\]
appearing in the proof of Lemma~\ref{lem:invop} is a bijection, known as the \textbf{symmetry at $y$} and sometimes denoted $S_y$. We may regard $\rght{y}$ as an element of the symmetric group on $R$, $S_R$, and this gives us an action of the cyclic group $\langle \rght{y}\rangle$ on $R$.

Let $x,y\in R$, and suppose that the orbit of $x$ under the action of $\langle \rght{y}\rangle$ on $R$ is finite; that is, suppose that there is a positive integer $n(x,y)$ such that $x\rght{y}^{n(x,y)}=x$. Then 
\[
(x\rght{y}^{n(x,y)-1})\rght{y} = x = (x\rght{y}^{-1})\rght{y},
\]
and since $\rght{y}$ is a bijection it follows that $x\rght{y}^{n(x,y)-1}=x\rght{y}^{-1}$. We use this to prove the following sufficient condition on $(R,\rt)$ for every 
$\rt$-congruence on $R$ to be a $\lt$-congruence:

\begin{proposition}
\label{prop:rtimpliesrack}
Let $(R,\rt)$ be a rack, and suppose that for all $y\in R$, every orbit of the action of $\langle \rght{y}\rangle$ on $R$ is finite; that is, suppose that for each $x,y\in R$, there exists a positive integer $n(x,y)$ such that $x\rght{y}^{n(x,y)}=x$. Then every $\rt$-congruence on $R$ is a $\lt$-congruence (and hence a rack congruence).
\end{proposition}

\begin{proof}
Let $\sim$ be a $\rt$-congruence on $R$, and let $a,b,c,d\in R$ be such that $a\sim c$ and $b\sim d$. Then
\[
a\rght{b} = a\rt b \sim c\rt d = c\rght{d},
\]
so $a\rght{b} \sim c\rght{d}$. Inductively, this gives $a\rght{b}^k \sim c\rght{d}^k$
for all positive integers $k$. 

Let $n(a,b)$, $n(c,d)$ be such that $a\rght{b}^{n(a,b)}=a$ and $c\rght{d}^{n(c,d)}=c$, and let $N=\lcm(n(a,b),n(c,d))$. We may assume without loss of generality that $n(a,b),n(c,d)\geq2$, so then $N\geq2$ also and satisfies
\begin{align*}
a\rght{b}^N &= a, & c\rght{d}^N &= c.
\end{align*}
By our discussion above we have
\[
a\lt b = a\rght{b}^{-1} = a\rght{b}^{N-1} \sim c\rght{d}^{N-1} = c\rght{d}^{-1} = 
c\lt d,
\]
so $a\lt b\sim c\lt d$. It follows that $\sim$ is a $\lt$-congruence.
\end{proof}

The following results are immediate consequences of Proposition~\ref{prop:rtimpliesrack}.

\begin{corollary}
Let $(R,\rt)$ be an $n$-rack. Then every $\rt$-congruence on $R$ is a rack congruence.
\end{corollary}

\begin{corollary}
Let $(R,\rt)$ be a finite rack. Then every $\rt$-congruence on $R$ is a rack congruence.
\end{corollary}

\section{Rack and quandle homomorphisms}
\label{sec:homomorphisms}

In this section we define rack homomorphisms, show that every rack homomorphism defines a rack congruence, and prove a First Isomorphism Theorem for racks and quandles, analogous to the corresponding theorems for groups and rings.

\begin{definition}
Let $(R,\rt)$ and $(S,\rtb)$ be racks. A function $\phi:R\to S$ is a \textbf{rack homomorphism} if 
\[
\phi(x\rt y) = \phi(x)\rtb\phi(y)
\]
for all $x,y\in R$.
\end{definition}

The definition of a rack homomorphism requires it to respect the primary rack operation only. We now show that a rack homomorphism necessarily respects the inverse rack operation also. Note that the proof of this result depends on the fact that the target $(S,\rtb)$ is a rack.

\begin{lemma}
\label{lem:homomorphism}
Let $(R,\rt)$ and $(S,\rtb)$ be racks, and let $\phi:R\to S$ be a rack homomorphism. Then
\[
\phi(x\lt y) = \phi(x)\ltb\phi(y)
\]
for all $x,y\in R$.
\end{lemma}

\begin{proof}
We have
\begin{align*}
\phi(x\lt y) &= \bigl(\phi(x\lt y)\rtb\phi(y)\bigr)\ltb\phi(y) \\
             &= \phi((x\lt y)\rt y)\ltb\phi(y) \\
             &= \phi(x)\ltb\phi(y),
\end{align*}
as required.
\end{proof}

In groups and rings the notions of homomorphisms and quotients are connected by the corresponding First Isomorphism Theorems. We prove that the analogous connection holds for racks and quandles too. To prove this we begin by showing that every rack homomorphism defines a rack congruence:

\begin{proposition}
\label{prop:simphi}
Let $(R,\rt)$ and $(S,\rtb)$ be racks, and let $\phi:R\to S$ be a rack homomorphism. The relation $\sim_\phi$ defined on $R$ by
\begin{equation}
\label{eq:kernel}
x\sim_\phi y \qquad\Leftrightarrow\qquad \phi(x)=\phi(y)
\end{equation}
is a rack congruence on $R$.
\end{proposition}

For any sets $A$ and $B$ and function $\phi:A\to B$, the relation $\sim_\phi$ defined on $A$ by equation~\eqref{eq:kernel} is the \textbf{kernel} of $\phi$.

\begin{proof}
It is easily seen that $\sim_\phi$ is an equivalence relation, so we check the congruence conditions. Suppose that $a\sim_\phi c$ and $b\sim_\phi d$. Then
\[
\phi(a\rt b) = \phi(a)\rtb\phi(b) = \phi(c)\rtb\phi(d) = \phi(c\rt d),
\]
so $a\rt b \sim_\phi c\rt d$; and likewise, using Lemma~\ref{lem:homomorphism} we have
\[
\phi(a\lt b) = \phi(a)\ltb\phi(b) = \phi(c)\ltb\phi(d) = \phi(c\lt d),
\]
so $a\lt b \sim_\phi c\lt d$.
\end{proof}

We next show that a push-forward or pullback of a subrack is again a subrack:

\begin{proposition}
Let $(R,\rt)$ and $(S,\rtb)$ be racks, and let $\phi:R\to S$ be a rack homomorphism.
\begin{enumerate}
\item
If $A$ is a subrack of $R$, then $\phi(A)$ is a subrack of $S$. Moreover, if $A$ is a quandle then so is $\phi(A)$.
\item
If $B$ is a subrack of $S$, then $\phi^{-1}(B)$ is a subrack of $R$.
\end{enumerate} 
\end{proposition}

\begin{proof}
By Proposition~\ref{prop:subracktest} a subset of a rack is a subrack if and only if it is closed under both rack operations.
\begin{enumerate}
\item
Let $s_1,s_2\in\phi(A)$. Then there exist $a_1,a_2\in A$ such that $\phi(a_i)=s_i$ for $i=1,2$, and
\begin{align*}
s_1\rtb s_2 &= \phi(a_1)\rtb\phi(a_2) = \phi(a_1\rt a_2)\in\phi(A), \\
s_1\ltb s_2 &= \phi(a_1)\ltb\phi(a_2) = \phi(a_1\lt a_2)\in\phi(A).
\end{align*}
This shows that $\phi(A)$ is a subrack of $S$. If moreover $A$ is a quandle, then
\[
s_1\rtb s_1 = \phi(a_1)\rtb\phi(a_1) = \phi(a_1\rt a_1) = \phi(a_1) = s_1, 
\]
so the idempotence axiom holds in $\phi(A)$.
\item
Let $r_1,r_2\in\phi^{-1}(B)$. Then $b_i=\phi(r_i)\in B$ for $i=1,2$, and
\begin{align*}
\phi(r_1\rt r_2) &= \phi(r_1)\rtb\phi(r_2) = b_1\rtb b_2\in B, \\
\phi(r_1\lt r_2) &= \phi(r_1)\ltb\phi(r_2) = b_1\ltb b_2\in B.
\end{align*}
This shows that $\phi^{-1}(B)$ is a subrack of $R$.
\end{enumerate}
\end{proof}

We now tie everything in this section together in the First Isomorphism Theorem for racks and quandles:

\begin{theorem}[First Isomorphism Theorem for Racks]
\label{thm:firstiso}
Let $(R,\rt)$ and $(S,\rtb)$ be racks, let $\phi:R\to S$ be a rack homomorphism, and let $\sim_\phi$ be the rack congruence on $R$ of Proposition~\ref{prop:simphi}. Then the map $\psi: R/{\sim_\phi}\to\phi(R)$ defined by
\[
\psi([x]) = \phi(x)
\]
is a rack isomorphism. 
\end{theorem}

\begin{proof}
First note that $\psi$ is well defined and one-to-one by definition of $\sim_\phi$, and it is onto by choice of domain and codomain. It remains to check that it is a rack homomorphism. Given $[x],[y]\in R/{\sim_\phi}$ we have
\begin{align*}
\psi([x]\rt[y]) &= \psi([x\rt y]) \\
                &= \phi(x\rt y) \\
                &= \phi(x)\rtb\phi(y) \\ 
                &= \psi([x])\rtb\psi([y]),
\end{align*}
completing the proof.
\end{proof}

\section{Existence of half congruences}
\label{sec:existence}

In Section~\ref{sec:congruences} we proved that an equivalence relation $\sim$ on a rack $R$ induces the structure of a rack in the quotient $R/{\sim}$ if and only if $\sim$ is both a $\rt$-congruence and a $\lt$-congruence. 
As we have seen there are conditions under which one congruence condition necessarily implies the other.
We now show however that in general both congruence conditions are necessary, by constructing explicit examples of half congruences.

We will construct our first example in a proper rack (a rack that is not a quandle), and then we will modify the construction to obtain an example in a quandle. By restricting our attention to a subquandle of this example we will obtain an example of a half congruence in a finitely presented quandle.

\subsection{Example in a proper rack}

We will construct our first example of a half congruence in the rack $\binseq_\ell$. This is a proper rack, in the sense that it is a rack that is not a quandle.

Recall from Example~\ref{ex:leftshift} that $\binseq_\ell$ is the constant action rack with underlying set $\binseq=\{0,1\}^\integer$, and rack operations
\begin{align*}
a\rt b &= \ell(a), &
a\lt b &= r(a),
\end{align*}
where $\ell$ and $r$ are the left- and right-shift operators, respectively. 
Define the relation $\sim$ on $\binseq$ by
\[
a\sim b \qquad\Leftrightarrow\qquad\text{$a_i=b_i$ for all $i\geq0$}.
\]
It is easily seen that $\sim$ is an equivalence relation on \binseq, and moreover that 
\[
a\sim b \qquad\Rightarrow\qquad \ell(a)\sim\ell(b).
\]
It follows that $\sim$ is a $\rt$-congruence on $\binseq_\ell$.
However, $\sim$ is not a $\lt$-congruence on $\binseq_\ell$, because
\[
a\sim b \qquad\not\Rightarrow\qquad r(a)\sim r(b).
\]
For a concrete example, define $a,b\in\binseq$ by $a_i=0$ for all $i$ and $b_{-1}=1$, $b_{i}=0$ for all $i\neq -1$. Then $a\sim b$ but $r(a)\not\sim r(b)$, because $(r(a))_0=0$ while $(r(b))_0=1$.

\subsection{Example in a quandle}

To construct an example of a half congruence in a quandle we modify the rack operation in $\binseq_\ell$ to obtain a quandle structure on \binseq. We first define a second equivalence relation $\simeq$ on \binseq\ by 
\[
a\simeq b \qquad\Leftrightarrow\qquad \text{there exist $s,t\in\integer$ such that $\ell^s(a)\sim\ell^t(b)$}.
\]
Equivalently,
\[
a\simeq b \qquad\Leftrightarrow\qquad \text{there exist $j,k\in\integer$ s.t.\ $a_i=b_{i+j}$ for all $i\geq k$}.
\]
The equivalence relation $\simeq$ is a rack congruence on $\binseq_\ell$, although we will not make use of this fact. Observe that $\sim$ is a refinement of $\simeq$, in the sense that $a\sim b$ implies $a\simeq b$; and that $c\simeq \ell(c)\simeq r(c)$ for all $c\in\binseq$, so $a\simeq b\Leftrightarrow a\simeq\ell(b)\Leftrightarrow a\simeq r(b)$.

Now define binary operations $\rtb,\ltb$ on \binseq\ by
\begin{align*}
a \rtb b &= \begin{cases}
             a & a\simeq b, \\
             \ell(a) & a\not\simeq b, 
             \end{cases} &
a \ltb b &= \begin{cases}
             a & a\simeq b, \\
             r(a) & a\not\simeq b.
             \end{cases} 
\end{align*} 
We verify the quandle axioms for $(\binseq,\rtb)$ to show that it is a quandle with inverse quandle operation $\ltb$.
\begin{description}
\item[Idempotence] 
For all $a\in\binseq$ we have $a\rtb a=a$, since $a\simeq a$ for all $a$.
\item[Right invertibility]
We will  use Lemma~\ref{lem:invop} to check right invertibility, by checking that $\ltb$ is the inverse operation to $\rtb$. For all $a,b\in\binseq$ we have
\begin{align*}
(a\rtb b)\ltb b &= \begin{cases}
                 a\ltb b & a\simeq b, \\
                 \ell(a)\ltb b & a\not\simeq b, 
                 \end{cases} \\
               &= \begin{cases}
                  a & a\simeq b, \\
                  r(\ell(a)) & a\not\simeq b
                  \end{cases} \\
               &= a.                 
\end{align*}
The calculation $(a\ltb b)\rtb b=a$ is similar.
\item[Right self-distributivity]
For all $a,b,c\in\binseq$ we have
\begin{align*}
(a\rtb b)\rtb c &= \begin{cases}
                 a\rtb c & a\simeq b, \\
                 \ell(a)\rtb c & a\not\simeq b, 
                 \end{cases} \\
               &= \begin{cases}
                  a         & a\simeq b, a\simeq c, \\
                  \ell(a)   & a\simeq b, a\not\simeq c, \\
                  \ell(a)   & a\not\simeq b, a\simeq c, \\
                  \ell^2(a)   & a\not\simeq b, a\not\simeq c,                  
                  \end{cases}
\end{align*}
while
\begin{align*}
(a\rtb c)\rtb (b\rtb c)
                  & = \begin{cases}
                  a \rtb b            & a\simeq c, b\simeq c, \\
                  a \rtb \ell(b)      & a\simeq c, b\not\simeq c, \\
                  \ell(a)\rtb b       & a\not\simeq c, b\simeq c, \\
                  \ell(a)\rtb\ell(b)  & a\not\simeq c, b\not\simeq c.
                  \end{cases}
\end{align*}
To proceed, observe that the case $a\not\simeq c, b\not\simeq c$ is the only one where 
the relationship between $a$ and $b$ with respect to $\simeq$ is not determined by their relationships with $c$. We find
\begin{align*}
(a\rtb c)\rtb (b\rtb c)
                  & = \begin{cases}
                  a       & a\simeq c, b\simeq c, \\
                  \ell(a) & a\simeq c, b\not\simeq c,\\ 
                  \ell^2(a)   & a\not\simeq c, b\simeq c, \\
                  \ell(a)     & a\not\simeq c, b\not\simeq c, a\simeq b,           \\     
                  \ell^2(a)   & a\not\simeq c, b\not\simeq c, a\not\simeq b.               
                  \end{cases}
\end{align*}
The third and fifth cases together cover the case $a\not\simeq b, a\not\simeq c$, and their outcomes agree, so we may combine them into a single case. The remaining three cases of the case division can each be re-expressed in terms of the relationships with respect to $\simeq$ between $a$ and each of $b$ and $c$; doing this and reordering the cases we obtain finally
\[
(a\rtb c)\rtb (b\rtb c)
               = \begin{cases}
                  a         & a\simeq b, a\simeq c, \\
                  \ell(a)   & a\simeq b, a\not\simeq c, \\
                  \ell(a)   & a\not\simeq b, a\simeq c, \\
                  \ell^2(a)   & a\not\simeq b, a\not\simeq c.                
                  \end{cases}
\]
This agrees with the result for $(a\rtb b)\rtb c$ found above, completing the proof that $\rtb$ is right self-distributive.
\end{description}

We now show that $\sim$ is a half congruence on $(\binseq,\rtb)$, by showing that it is a $\rtb$-congruence but not a $\ltb$-congruence. To show that $\sim$ is a $\rtb$-congruence we will make use of the following lemma:

\begin{lemma}
\label{lem:simandsimeq}
Let $x,y,z\in\binseq$ be such that $x\sim y$. Then $x\simeq z$ if and only if $y\simeq z$.
\end{lemma}

\begin{proof}
By symmetry it suffices to show that $x\simeq z$ implies $y\simeq z$.
Let $s,t\in\integer$ be such that $\ell^s(x)\sim\ell^t(z)$. Then we have $x_{i+s}=z_{i+t}$ for $i\geq 0$, and $x_i=y_i$ for $i\geq0$. Let $r=\max\{0,-s\}$. Then $r\geq0$, so if $i\geq0$ then $i+r\geq 0$, which implies $x_{i+r+s}=z_{i+r+t}$; and $r+s\geq0$, so if $i\geq0$ then $i+r+s\geq0$, which implies $x_{i+r+s}=y_{i+r+s}$. Hence $y_{i+r+s}=z_{i+r+t}$ for $i\geq0$, which says $\ell^{r+s}(y)\sim\ell^{r+t}(z)$. Therefore $y\simeq z$.
\end{proof}

\begin{proposition}
The equivalence relation $\sim$ is a $\rtb$-congruence on \binseq\ that is not a $\ltb$-congruence.
\end{proposition}

\begin{proof}
We first show that $\sim$ is a $\rtb$-congruence on \binseq. Suppose that $a,b,c,d\in\binseq$ with $a\sim c$ and $b\sim d$. Applying Lemma~\ref{lem:simandsimeq} twice, first with $(x,y,z)=(a,c,b)$ and then with $(x,y,z)=(b,d,c)$, we conclude that $a\simeq b$ if and only if $c\simeq d$. If $a\simeq b$ then
\[
a\rtb b = a \sim c = c\rtb d;
\]
and if $a\not\simeq b$ then
\[
a\rtb b = \ell(a) \sim \ell(c) = c\rtb d.
\]
In either case $a\rtb b \sim c\rtb d$, so $\sim$ is a $\rtb$-congruence.

We now show that $\sim$ is not a $\ltb$-congruence on \binseq. Define $\alpha,\beta,\gamma\in\binseq$ by
\begin{align*}
\alpha_i &= \begin{cases} 
            1 & i=0, \\
            0 & i\neq 0,
            \end{cases} &
\beta_i  &= \begin{cases} 
            1 & i\leq 0, \\
            0 & i> 0,
            \end{cases} &
\gamma_i &= \text{$1$ for all $i$}.
\end{align*}
Then $\alpha\sim\beta$, but $\alpha\ltb\gamma=r(\alpha)\not\sim r(\beta)=\beta\ltb\gamma$.
\end{proof}

\subsection{A finitely presented example}

By restricting attention to the subquandle of $(\binseq,\rtb)$ generated by $\alpha,\beta,\gamma$ we obtain a finitely presented example. Let
\[
\binseq_0=\{\ell^j(\alpha),\ell^j(\beta),\gamma:j\in\integer\}.
\]
Then $(\binseq_0,\rtb)$ is a subquandle of $(\binseq,\rtb)$, because $\binseq_0$ is a union of orbits of $(\binseq,\rtb)$ and hence closed with respect to both $\rtb$ and $\ltb$. We claim that $(\binseq_0,\rtb)$ has presentation
\begin{equation}
\label{eq:B0presentation}
\langle \alpha,\beta,\gamma : \alpha\rtb\beta = \alpha, \beta\rtb\alpha = \beta, \gamma\rtb\alpha = \gamma, \gamma\rtb\beta = \gamma\rangle. 
\end{equation}

The relations of~\eqref{eq:B0presentation} certainly hold in $(\binseq_0,\rtb)$. 
For the purpose of proving this claim let $(Q,\rt)$ be a quandle with 
elements $a,b,c$ satisfying the relations of~\eqref{eq:B0presentation} under the substitutions $\alpha\mapsto a$, $\beta\mapsto b$, $\gamma\mapsto c$; that is, suppose that $a,b,c$ satisfy the equations
\begin{align*}
a\rt b &= a, & b\rt a &= b, & c\rt a &= c, & c\rt b &= c.
\end{align*}
We will show that $(\binseq_0,\rtb)$ satisfies the required universal property of the quandle defined by~\eqref{eq:B0presentation}, by showing that
the set map $f:\{\alpha,\beta,\gamma\}\to\{a,b,c\}$ given by the substitutions above uniquely extends to a quandle homomorphism $(\binseq_0,\rtb)\to(Q,\rt)$.  

Let $S$ be the subquandle of $Q$ generated by $a,b,c$, and for
$x\in\{a,b\}$ and $j\in\integer$ let $x^j=x\rght{c}^j$; inductively,  $x^j$ is given by
\[
x^j = \begin{cases}
      x & j=0, \\
      x^{j-1}\rt c & j>0, \\
      x^{j+1}\lt c & j<0.
      \end{cases}
\]
We claim that $S=X$, where $X=\{a^j,b^j,c:j\in\integer\}$. Since the set $X$ contains the generators $a,b,c$ of $S$, and any subquandle containing $a,b,c$ must contain $X$, to prove this claim it is sufficient to show that $X$ is a subquandle, by showing that it is closed under the quandle operations. In doing so we will also completely determine the multiplication table in $(S,\rt)$.
Note that we explicitly make no assumption that different expressions of the form $a^j$, $b^j$, $c$ represent distinct elements of $Q$. 

Observe that for all $j\in\integer$ we have $x^j\rt c = x^{j+1}\in X$ and $x^j\lt c=x^{j-1}\in X$. To show that $c\rt x^j\in X$ for all $j$ we prove by induction that $c\rt x^j=c$ for all $j$. The base case $j=0$ is given by the presentation, and for $j>0$ we have
\begin{align*}
c\rt x^j &= c\rt (x^{j-1}\rt c) \\
         &= (c\rt c)\rt (x^{j-1}\rt c) & &\text{(idempotence)} \\
         &= (c\rt x^{j-1})\rt c && \text{(right self-distributivity)} \\
         &= c\rt c && \text{(inductive hypothesis)} \\
         &= c.
\end{align*} 
The case $j<0$ follows similarly, using the right distributivity of $\lt$ over $\rt$ (Lemma~\ref{lem:distributivity}):
\begin{align*}
c\rt x^j &= c\rt (x^{j+1}\lt c) \\
         &= (c\lt c)\rt (x^{j+1}\lt c) & &\text{(idempotence)} \\
         &= (c\rt x^{j+1})\lt c && \text{(right distributivity)} \\
         &= c\lt c && \text{(inductive hypothesis)} \\
         &= c.
\end{align*}
Using the result $c\rt x^j=c$ just proved we now have
\begin{equation}
\label{eq:fixed}
c\lt x^j = (c\rt x^j)\lt x^j = c\in X.
\end{equation}

For the remaining products we show that $x^j\rt y^k=x^j$ and $x^j\lt y^k=x^j$ for $x,y\in\{a,b\}$ and all $j,k\in\integer$. As seen in equation~\eqref{eq:fixed} $u\rt v=u$ implies $u\lt v = u$, so it is enough to show that $x^j\rt y^k=x^j$. This is done by induction on $j$ and then $k$, with the idempotence axiom or the presentation providing the base case $j=k=0$, according to whether or not $x=y$. 
Inducting on $j>0$ we have 
\begin{align*}
x^j\rt y &= (x^{j-1}\rt c)\rt y \\
         &= (x^{j-1}\rt y)\rt (c\rt y) && \text{(distributivity)} \\
         &= x^{j-1}\rt c && \text{(inductive hypothesis)} \\
         &= x^j;
\end{align*} 
a similar calculation for $j<0$ establishes $x^j\rt y = x^j$ for all $j$. 
Using this and inducting on $k>0$ we then have
\begin{align*}
x^j\rt y^k &= (x^{j-1}\rt c)\rt (y^{k-1}\rt c) \\
         &= (x^{j-1}\rt y^{k-1})\rt c && \text{(distributivity)} \\
         &= x^{j-1}\rt c && \text{(inductive hypothesis)} \\
         &= x^j.
\end{align*}
A similar calculation for $k<0$ completes the proof that 
$x^j\rt y^k=x^j$ for $x,y\in\{a,b\}$ and all $j,k\in\integer$.

Our work above completes the proof that $S=X$, and shows moreover that the quandle operation in $S$ is given by
\begin{align*}
a^j \rt a^k &= a^j, & a^j \rt b^k &= a^j, & a^j \rt c &= a^{j+1}, \\
b^j \rt a^k &= b^j, & b^j \rt b^k &= b^j, & b^j \rt c &= b^{j+1}, \\
c \rt a^k   &= c,   & c \rt b^k      &= c,   &   c \rt c &= c.
\end{align*}
Now consider the map $\phi:\binseq_0\to Q$ given by
\begin{align*}
\ell^j(\alpha) &\mapsto a^j, \\
\ell^j(\beta)  &\mapsto b^j, \\
\gamma &\mapsto c.
\end{align*}
This map is well defined, because for $\mu,\nu\in\{\alpha,\beta\}$ and $j,k\in\integer$ we have $\ell^j(\mu)=\ell^k(\nu)$ if and only if $\mu=\nu$ and $j=k$. It extends our set map $f:\{\alpha,\beta,\gamma\}\to\{a,b,c\}$, and using the table of quandle operations in $S$ above we can see that it is a quandle homomorphism. Moreover, any quandle homomorphism $\binseq_0\to Q$ extending $f$ is uniquely determined by the images $f(\alpha),f(\beta),f(\gamma)$, because $\{\alpha,\beta,\gamma\}$ is a generating set for $\binseq_0$. This completes the proof that $(\binseq_0,\rtb)$ has the required universal property of the quandle with presentation~\eqref{eq:B0presentation}.

\section{Congruences in weighted average quandles on $\rational$}
\label{sec:weightedaverage}

Our goal in this section is to completely characterise congruences in weighted average quandles on $\rational$. This will give us an infinite family of quandles $\mathcal{F}$ such that every quandle $Q\in\mathcal{F}$ has nontrivial $\rt$-congruences, nontrivial $\lt$-congruences, and nontrivial quandle congruences; and moreover, for each of the following mutually exclusive properties there exists a quandle $Q\in\mathcal{F}$ that satisfies the given property:
\begin{enumerate}
\item
Every $\rt$-congruence on $Q$ is a $\lt$-congruence, and conversely.
\item
Every $\rt$-congruence on $Q$ is a $\lt$-congruence, but there exist $\lt$-congruences on $Q$ that are not $\rt$-congruences.
\item
Every $\lt$-congruence on $Q$ is a $\rt$-congruence, but there exist $\rt$-congruences on $Q$ that are not $\lt$-congruences.
\item
There exists a $\rt$-congruence on $Q$ that is not a $\lt$-congruence, and a $\lt$-congruence on $Q$ that is not a $\rt$-congruence.
\end{enumerate}

The characterisation is given by the following theorem.

\begin{theorem}[Characterisation of congruences in \avgquandle\tau]
\label{thm:weightedavgcongruences}
Let $\avgquandle{\tau}$ be a nontrivial weighted average quandle on $\rational$ with weight $\tau\neq0,1$.
\begin{enumerate}
\item
\label{item:rtclassification}
A relation $\sim$ on \rational\ is a $\rt$-congruence on \avgquandle{\tau}\ if and only if there exists a subgroup $D$ of \rational\ that is closed under multiplication by $\tau$ such that
\[
x\sim y \qquad\Leftrightarrow\qquad y-x\in D.
\]
\item
\label{item:ltclassification}
A relation $\sim$ on \rational\ is a $\lt$-congruence on \avgquandle{\tau}\ if and only if there exists a subgroup $D$ of \rational\ that is closed under multiplication by $\tau^{-1}$ such that
\[
x\sim y \qquad\Leftrightarrow\qquad y-x\in D.
\]
\end{enumerate}
\end{theorem}

Since the dual quandle to the weighted average quandle \avgquandle\tau\ is the weighted average quandle \avgquandle{\tau^{-1}}, and taking the dual exchanges the roles of $\rt$ and $\lt$, it suffices to prove part~\eqref{item:rtclassification} of Theorem~\ref{thm:weightedavgcongruences}. Part~\eqref{item:ltclassification} then follows immediately as a corollary. We begin by proving that a subgroup of \rational\ that is closed under multiplication by $\tau$ defines a $\rt$-congruence on \avgquandle{\tau}.

\begin{proposition}
\label{prop:congruencefromZtsubmodule}
Let $\avgquandle{\tau}$ be a nontrivial weighted average quandle on $\rational$ with weight $\tau\neq0,1$, and let $D$ be a subgroup of \rational\ that is closed under multiplication by $\tau$. Then the relation $\sim$ on \rational\ defined by 
\[
x\sim y \qquad\Leftrightarrow\qquad y-x\in D
\]
is a $\rt$-congruence on \avgquandle{\tau}.
\end{proposition}

\begin{proof}
The relation $\sim$ is the relation defining the cosets of $D$. This is known or easily checked to be an equivalence relation, so we verify the $\rt$-congruence condition.

Let $a,b,c,d\in\rational$ with $a\sim c$ and $b\sim d$. Then 
\begin{align*}
(c\rt d)-(a\rt b) &= (\tau c + (1-\tau)d)-(\tau a +(1-\tau)b) \\
                  &= \tau(c-a)+(1-\tau)(d-b),
\end{align*}
which belongs to $D$ because $c-a,d-b\in D$, and $D$ is a subgroup closed under multiplication by $\tau$. We conclude that $a\rt b\sim c\rt d$, so $\sim$ is a $\rt$-congruence, as claimed.
\end{proof}

We now turn to the more involved task of showing that every $\rt$-congruence 
on \avgquandle\tau\ comes from a subgroup of \rational\ that is closed under multiplication by $\tau$.

\begin{proposition}
\label{prop:submodulefromcongruence}
Let $\avgquandle{\tau}$ be a nontrivial weighted average quandle on \rational\ with weight $\tau\neq0,1$, and let $\sim$ be a $\rt$-congruence. Let 
\[
D=\{y-x:x,y\in\rational,x\sim y\}.
\]
Then $D$ is a subgroup of \rational\ that is closed under multiplication by $\tau$, and  
\[
x\sim y \qquad\Leftrightarrow\qquad y-x\in D.
\]
\end{proposition}

\begin{proof}
Clearly $D\neq\emptyset$, because $a\sim a$ for all $a$ and so $0=a-a\in D$. 
Given $d\in D$ let $x,y\in \rational$ be such that $x\sim y$ and $y-x=d$. Given $a\in\rational$, there exists $u\in\rational$ such that $(1-\tau)u=a-\tau x$; explicitly, we have of course
\begin{equation}
\label{eq:explicitsolution}
u = \frac{a-\tau x}{1-\tau}.
\end{equation}
Then
\begin{align*}
x\rt u &= \tau x +(1-\tau)u=\tau x+a-\tau x = a, \\
y\rt u &= (x+d)\rt u \\
       &=\tau(x+d) + (1-\tau)u \\
       &=\tau(x+d)+a-\tau x = a + \tau d.              
\end{align*}
Since $\sim$ is a $\rt$-congruence we have
$x\rt u\sim y\rt u$, and we conclude that $a\sim a+\tau d$ for all $a\in\rational$ and $d\in D$. This shows also that $\tau d\in D$ for all $d\in D$, that is, $D$ is closed under multiplication by $\tau$.

Next, let
\[
v = \tau^{-1}(a-(1-\tau)x).
\]
Then 
\begin{align*}
v \rt x &= \tau\tau^{-1}(a-(1-\tau)x)+(1-\tau)x = a, \\
v \rt y &= v \rt (x+d)  \\
        &= \tau\tau^{-1}(a-(1-\tau)x)+(1-\tau)(x+d) = a+(1-\tau)d.
\end{align*}
Since $\sim$ is a $\rt$-congruence we have
$v\rt x\sim v\rt y$, and we conclude that $a\sim a+(1-\tau)d$ for all $a\in\rational$ and $d\in D$.

Combining the results just proved we obtain
\[
a \sim a+\tau d \sim (a+\tau d)+(1-\tau)d = a+d.
\]
Hence $a\sim a+d$ for all $a\in\rational$ and $d\in D$. Conversely, if $a\sim b$ for some $a,b\in\rational$ then $b-a\in D$ by definition of $D$. We conclude that
\[
x\sim y \qquad\Leftrightarrow\qquad y-x\in D,
\]
as claimed.

To complete the proof we show that $D$ is a subgroup, by showing that it is closed under subtraction. Given $d_1,d_2\in D$, take $x\in\rational$. Then from above we have $x\sim x+d_1$ and $x\sim x+d_2$, so $x+d_1\sim x+d_2$ by symmetry and transitivity. Hence
\[
d_2-d_1 = (x+d_2)-(x+d_1) \in D,
\]
showing that $D$ is a subgroup. 
\end{proof}

\begin{remark}
Notice that the calculation following~\eqref{eq:explicitsolution} does not require the explicit value of $u$, but only the existence of a solution $u$ to the equation $(1-\tau)u = a-\tau x$. The remainder of the argument uses only the fact that $\tau$ is a unit. Thus the proof goes through as written to prove the corresponding result in any Alexander quandle where $1-t$ acts surjectively. We discuss this in Section~\ref{sec:generalAlexanderquandles} below.
\end{remark}

Closure of a subgroup under multiplication by $\tau$ can be more conveniently expressed in terms of the denominator of $\tau$:

\begin{lemma}
\label{lem:closure}
Let $\tau\in\rational$ with $\tau\neq 0$. Write $\tau=\frac{p}q$, where $p,q\in\integer$ with $\gcd(p,q)=1$. A subgroup $D$ of \rational\ is closed under multiplication by $\tau$ if and only if it is closed under multiplication by $\frac1q$.
\end{lemma}

\begin{proof}
If $D$ is closed under multiplication by $\frac1q$, then it is certainly closed under multiplication by $\tau=p\cdot\frac1q$. So suppose $D$ is closed under multiplication by $\tau$. Since $\gcd(p,q)=1$ there exist $k,\ell\in\integer$ such that $kp+\ell q=1$. Given $d\in D$ we have
\[
\frac1q\cdot d = \left(k\frac{p}q+\ell\right)d = k\tau d +\ell d\in D,
\]
so $D$ is closed under multiplication by $\frac1q$.
\end{proof}

We can now show that the family of quandles 
\[
\mathcal{F}=\{\avgquandle\tau: \tau\in\rational,\tau\neq0,1\}
\]
has the properties listed at the start of the section.

\begin{theorem}
Let $\avgquandle\tau$ be a nontrivial weighted average quandle on \rational\ with weight $\tau\neq0,1$. Then $\avgquandle\tau$ has nontrivial $\rt$-congruences, nontrivial $\lt$-congruences, and nontrivial quandle congruences. Moreover:
\begin{enumerate}
\item
If $\tau=-1$ then a relation on $\rational$ is a $\rt$-congruence if and only if it is a $\lt$-congruence.
\item
If $\tau^{-1}\in\integer$, $\tau\neq-1$, then every $\rt$-congruence on $\avgquandle\tau$ is a $\lt$-congruence, but not every $\lt$-congruence is a $\rt$-congruence.
\item
If $\tau\in\integer$, $\tau\neq-1$, then every $\lt$-congruence on $\avgquandle\tau$ is a $\rt$-congruence, but not every $\rt$-congruence is a $\lt$-congruence.
\item
If $\tau,\tau^{-1}\notin\integer$ then there exists a $\rt$-congruence on \avgquandle\tau\ that is not a $\lt$-congruence, and a $\lt$-congruence that is not a $\rt$-congruence.
\end{enumerate}
\end{theorem}

\begin{proof}
Write $\tau=\frac{p}q$ with $p,q\in\integer$ and $\gcd(p,q)=1$. By Theorem~\ref{thm:weightedavgcongruences} and Lemma~\ref{lem:closure}, $\rt$-congruences on \avgquandle\tau\ correspond to subgroups of \rational\ that are closed under multiplication by $\frac1q$; and $\lt$-congruences on \avgquandle\tau\ correspond to subgroups of \rational\ that are closed under multiplication by $\frac1p$.

Given a nonzero integer $r$, define
\[
D_r = \left\{\frac{k}{r^\ell}:k,\ell\in\integer,\ell\geq0\right\}.
\]
Then $D_r$ is a subgroup of $\rational$ that is closed under multiplication by $\frac1r$. Taking $r=q$ shows that $\avgquandle\tau$ has nontrivial $\rt$-congruences; taking $r=p$ shows that $\avgquandle\tau$ has nontrivial $\lt$-congruences; and similarly the subgroup
\[
D_{p,q} = \left\{\frac{k}{p^\ell q^m}:k,\ell,m\in\integer,\ell,m\geq0\right\}
\]
shows that $\avgquandle\tau$ has nontrivial quandle congruences.

We now consider each of the cases given in the theorem in turn.
\begin{enumerate}
\item
If $\tau=-1$ then $\tau=\tau^{-1}$, so by Theorem~\ref{thm:weightedavgcongruences} every $\rt$-congruence is a $\lt$-congruence, and conversely. (This also follows more simply from the fact that $\rt$ and $\lt$ coincide for $\tau=-1$.)
\item
If $\tau^{-1}\in\integer$ with $\tau\neq-1$ then $|p|=1$ and $|q|>1$. Every subgroup of $\rational$ is closed under multiplication by $\frac1p$, so every 
$\rt$-congruence on $\avgquandle\tau$ is a $\lt$-congruence; but the subgroup $\integer$ of \rational\ is closed under multiplication by $\frac1p$ but not $\frac1q$, so defines a $\lt$-congruence that is not a $\rt$-congruence.
\item
If $\tau\in\integer$ with $\tau\neq-1$ then $|q|=1$ and $|p|>1$. The conclusion follows as in the previous case, with the roles of $p$ and $q$ reversed.
\item
If $\tau,\tau^{-1}\notin\integer$ then $|p|,|q|>1$. The subgroup $D_q$ is closed under multiplication by $\frac1q$ but not $\frac1p$, so defines a $\rt$-congruence that is not a $\lt$-congruence; while the subgroup $D_p$ is closed under multiplication by $\frac1p$ but not $\frac1q$, so defines a $\lt$-congruence that is not a $\rt$-congruence.
\end{enumerate}
\end{proof}

\section{Discussion: Congruences in general Alexander quandles}
\label{sec:generalAlexanderquandles}

We briefly discuss the extent to which the results of Section~\ref{sec:weightedaverage} extend to general Alexander quandles. Our work yields a complete characterisation of quandle congruences in connected Alexander quandles.

In Proposition~\ref{prop:congruencefromZtsubmodule} we showed that a subgroup $D$ of \rational\ that is closed under multiplication by $\tau$ defines a $\rt$-congruence in \avgquandle{\tau}. This result may be restated as saying that a $\zt$-submodule of $\avgquandle{\tau}$ defines a $\rt$-congruence in $\avgquandle{\tau}$.

The proof of Proposition~\ref{prop:congruencefromZtsubmodule} goes through unchanged to prove the analogous result in a general Alexander quandle:
\begin{proposition}
\label{prop:congruence-from-submodule}
Let $A$ be an Alexander quandle, and let $D$ be a \zt-submodule of $A$. Then the relation $\sim$ defined on $A$ by
\begin{equation}
\label{eq:cosetrelation}
x\sim y \qquad\Leftrightarrow\qquad y-x\in D
\end{equation}
is a $\rt$-congruence on $A$.
\end{proposition}

It then follows that a $\integer[t^{-1}]$-submodule of $A$ defines a $\lt$-congruence on $A$, and a \zttinv-submodule defines a quandle congruence on $A$. Moreover, it is easily checked that a quotient quandle $A/{\sim}$ of an Alexander quandle is again an Alexander quandle with respect to the induced $\zttinv$-module structure given by
\begin{align*}
[x]+[y] &= [x+y], & p(t)[x] = [p(t)x],
\end{align*}
if and only if $\sim$ is the coset relation of a \zttinv-submodule $D$ of $A$.

In Proposition~\ref{prop:submodulefromcongruence} we further showed that every $\rt$-congruence in $\avgquandle{\tau}$ comes from a $\zt$-submodule of $\avgquandle{\tau}$. However, the proof of this result used the fact that $(1-\tau)\rational=\rational$, and so to prove the corresponding result in an Alexander quandle $A$ we require $1-t\in\zttinv$ to act surjectively on $A$. 
The proof of Proposition~\ref{prop:submodulefromcongruence} goes through essentially as written to give us the following result:
\begin{proposition}
\label{prop:congruences-in-connected-case}
Let $A$ be an Alexander quandle such that $(1-t)A=A$, and let $\sim$ be a $\rt$-congruence on $A$. Let 
\[
D=\{y-x:x,y\in A,x\sim y\}.
\]
Then $D$ is a $\zt$-submodule of $A$, and
\[
x\sim y \qquad\Leftrightarrow\qquad y-x\in D.
\]
\end{proposition}
If $Q$ is a quandle and $x\in Q$, then the \textbf{orbit} of $x$ is the orbit of $x$ under the action of the translation group $\langle \rght{y} : y\in Q\rangle$, the subgroup of the symmetric group $S_Q$ generated by the right translation maps. A quandle is \textbf{connected} if it consists of a single orbit. In the case of Alexander quandles it is a known fact (see for example~\cite{hulpke-et-al-2016}) that the orbit of $x\in A$ is the coset $x+(1-t)A$. This is easily verified as follows: for any $a,x,y\in A$ we have
\begin{align*}
(x+(1-t)a)\rt y &= t(x+(1-t)a)+(1-t)y \\ 
                &= x+(1-t)(ta+y-x), \\
(x+(1-t)a)\lt y &= t^{-1}(x+(1-t)a)+(1-t^{-1})y  \\
                &= x+(1-t)t^{-1}(a+x-y),
\end{align*}
so $x+(1-t)A$ is invariant under the action of the translation group; and moreover
\[
(x\lt 0)\rt a = t(t^{-1}x) + (1-t)a = x+(1-t)a,
\]
so $x+(1-t)A$ is contained in the orbit of $x$. 
Therefore an Alexander quandle is connected if and only if $(1-t)A=A$. Thus, Propositions~\ref{prop:congruence-from-submodule}
and~\ref{prop:congruences-in-connected-case} together characterise quandle congruences in connected Alexander quandles: they correspond precisely to the coset relations of \zttinv-submodules.

We conclude with an example of a quandle congruence in an Alexander quandle that is not connected that does not arise from a $\zttinv$-submodule.

\begin{example}
Let $A=\zttinv$ with the natural Alexander quandle structure that comes from regarding \zttinv\ as a module over itself in the natural way. Let $\equiv$ be an arbitrary equivalence relation on \integer, and define the equivalence relation $\sim$ on $A$ by
\[
f\sim g \qquad \Leftrightarrow\qquad f(1)\equiv g(1).
\]
For any $f,g\in A$ we have 
\[
(f\rt g)(1) = (tf+(1-t)g)(1) = f(1),
\]
and similarly $(f\lt g)(1)=f(1)$. Therefore $f\rt g\sim f\sim f\lt g$ for all $f,g\in A$, and it follows that $\sim$ is a quandle congruence on $A$. 
However, $\equiv$ can be chosen so that the equivalence classes of $\sim$ are not the cosets of any \zttinv-submodule of $A$: for a concrete example, we may define $\equiv$ on \integer\ by $x\equiv y$ if and only if $x=y$ or $\{x,y\}=\{0,1\}$.

As a second approach to this example we may proceed as follows. Let $Z=\{f\in\zttinv : f(1)=0\}$, the kernel of the ring homomorphism $\zttinv\to\integer$ given by evaluation at $1\in\integer$. The ideal $Z$ is a \zttinv\-submodule of $A$, so by Proposition~\ref{prop:congruence-from-submodule} the equivalence relation $\simeq$ on $A$ defined by 
\begin{align*}
f\simeq g \qquad&\Leftrightarrow \qquad f-g\in Z \\
                &\Leftrightarrow \qquad f(1)=g(1)
\end{align*}
is a quandle congruence on \zttinv. The quotient quandle
$A/{\simeq}$ is easily seen to be isomorphic to \integer\ with the trivial quandle structure. Since the quandle structure on the quotient is trivial, any equivalence relation $\equiv$ on \integer\ will be a quandle congruence, and pulling $\equiv$ back to $A$ gives us the congruence $\sim$.
\end{example}

\end{document}